\documentclass[12pt]{article}

\usepackage[left=2cm,right=2cm,top=3cm,bottom=3cm]{geometry}

\usepackage{amsmath, amssymb, latexsym, amsthm}
\usepackage{fullpage}

\usepackage{tikz}
\usetikzlibrary{shapes}
\usetikzlibrary{calc,decorations.pathreplacing}

\newtheorem{theorem}{Theorem} 
\newtheorem{theorem*}{Theorem} 

\newtheorem{lemma}[theorem]{Lemma}

\theoremstyle{definition}

\newtheorem{question}{Question}

\theoremstyle{remark}

{\end{enumerate}}

\newcommand{\RR}{\mathbb{R}}

\newcommand{\CL}[1]{\left\lceil #1 \right\rceil}

\newcommand{\FL}[1]{\left\lfloor #1 \right\rfloor}

\expandafter\ifx\csname dplus\endcsname\relax \csname newbox\endcsname\dplus\fi
\expandafter\ifx\csname dplustemp\endcsname\relax
\csname newdimen\endcsname\dplustemp\fi
\setbox\dplus=\vtop{\vskip -8pt\hbox{%
    \kern .2em
    \special{pn 6}%
    \special{pa 0 50}%
    \special{pa 50 100}%
    \special{fp}%
    \special{pa 50 100}%
    \special{pa 100 50}%
    \special{fp}%
    \special{pa 100 50}%
    \special{pa 50 0}%
    \special{fp}%
    \special{pa 50 100}%
    \special{pa 50 0}%
    \special{fp}%
    \special{pa 50 0}%
    \special{pa 0 50}%
    \special{fp}%
    \special{pa 0 50}%
    \special{pa 100 50}%
    \special{fp}%
    \hbox{\vrule depth0.080in width0pt height 0pt}%
    \kern .7em
  }%
}%

\def\st{\colon\,}

\title{The List Distinguishing Number Equals the Distinguishing Number for Interval Graphs}
\author{Poppy Immel\footnotemark[1] \ and Paul S.\ Wenger\footnotemark[1]}

\begin{document}

\maketitle

\begin{abstract}
A {\it distinguishing coloring} of a graph $G$ is a coloring of the vertices so that every nontrivial automorphism of $G$ maps some vertex to a vertex with a different color.
The {\it distinguishing number} of $G$ is the minimum $k$ such that $G$ has a distinguishing coloring where each vertex is assigned a color from $\{1,\ldots,k\}$.
A {\it list assignment} to $G$ is an assignment $L=\{L(v)\}_{v\in V(G)}$ of lists of colors to the vertices of $G$.
A {\it distinguishing $L$-coloring} of $G$  is a distinguishing coloring of $G$ where the color of each vertex $v$ comes from $L(v)$.
The {\it list distinguishing number} of $G$ is the minimum $k$ such that every list assignment to $G$ in which $|L(v)|=k$ for all $v\in V(G)$ yields a distinguishing $L$-coloring of $G$.
We prove that if $G$ is an interval graph, then its distinguishing number and list distinguishing number are equal. 

\end{abstract}

\renewcommand{\thefootnote}{\fnsymbol{footnote}}
\footnotetext[1]{
School of Mathematical Sciences, Rochester Institute of Technology, Rochester, NY;
{\tt pgi8114@rit.edu, pswsma@rit.edu}.}
\renewcommand{\thefootnote}{\arabic{footnote}}

\baselineskip18pt

\section{Introduction}

All graphs considered in this paper are finite and simple.
We denote the vertex set of a graph $G$ by $V(G)$.
An {\it isomorphism} from a graph $G$ to a graph $H$ is an adjacency-preserving bijection from $V(G)$ to $V(H)$.
An {\it automorphism} is an isomorphism from a graph to itself.
A {\it k-coloring} of a graph $G$ is a labeling of the vertices $\phi: V(G) \rightarrow \{ 1,2,\ldots,k \}$ (henceforth we adopt the standard notation $[k]$ for $\{1,\ldots,k\}$). 
A k-coloring of $G$ is {\it distinguishing} if every nontrivial automorphism of $G$ maps some vertex to a vertex with a different color. 
In \cite{albertson96}, Albertson and Collins introduced the {\it distinguishing number} of a graph $G$, denoted $D(G)$, which is the minimum $k$ such that $G$ has a distinguishing $k$-coloring.

Distinguishing numbers have been studied on a wide variety of graphs including those with dihedral automorphism groups~\cite{albertson96}, cartesian products of graphs~\cite{albertson,FI,IJK,imrich06,KZ}, trees \cite{cheng06}, interval graphs \cite{cheng09}, and planar graphs~\cite{arvind08}.
Beyond finite graphs, distinguishing numbers have been studied on infinite graphs~\cite{imrich07}, vector spaces~\cite{KWZ}, and the unit sphere~\cite{FGHSW2,lombardi}.

A natural generalization for coloring parameters is to consider the implication of different vertices having different sets of available colors rather than all taking their color from the set $[k]$.
A {\it list assignment} on a graph $G$ is an assignment of a list of colors to each vertex in $G$; we denote the list assigned to vertex $v$ by $L(v)$ and the whole assignment as $L=\{L(v)\}_{v\in V(G)}$.
A {\it list-coloring} (respectively {\it $L$-coloring}) of $G$ is a coloring of the vertices of $G$ where the color of each vertex $v$ is taken from its list (respectively $L(v)$).
There is a long history of studying proper list colorings, which were introduced independently by Vizing~\cite{vizing} and Erd\H os, Rubin, and Taylor~\cite{ERT}.
In~\cite{ferrara11}, Ferrara, Felsch, and Gethner began the study of the list variant of the distinguishing number.
They defined the {\it list distinguishing number} of a graph $G$, denoted $D_{\ell}(G)$, to be the minimum $k$ such that given any list assignment $L=\{L(v)\}_{v\in V(G)}$ in which $|L(v)|=k$ for all $v\in V(G)$ there is a distinguishing $L$-coloring.
Their results included determining the list distinguishing number of graphs with dihedral automorphism groups and also cartesian products of cycles.
At the end of their paper they also asked the following question.
\begin{question}\label{FFG}
Does there exist a graph $G$ such that $D(G)\neq D_{\ell}(G)$?
\end{question}

Note that it is clear that $D(G)\le D_{\ell}(G)$.
Specifically, if $D_{\ell}(G)=k$, then the assignment of the list $[k]$ to each vertex yields a distinguishing coloring, so $D(G)\le k$.
Therefore a negative answer to Question~\ref{FFG} requires a general proof that $D_{\ell}(G)\le D(G)$ for all graphs.

In~\cite{ferrara13}, Ferrara, Gethner, Hartke, Stolee, and Wenger proved that $D_\ell(G) = D(G)$ when $G$ is a tree.
To do this, they adapted an enumerative technique for determining distinguishing numbers of trees that was developed independently by Cheng~\cite{cheng06} and Arvind and Devanur~\cite{nordicpaper}.
This technique consists of utilizing the structure of trees to recursively determine the number of distinguishing $k$-colorings of a tree; this value is the distinguishing number analogue of the chromatic polynomial.
The distinguishing number is the minimum positive integer value of $k$ for which the number of distinguishing $k$-colorings is positive.
Ferrara et al.\ modified this method to count the number of distinguishing $L$-colorings of a tree when it has been given a list assignment $L$.

In this paper, we use this enumerative method to study the list distinguishing number of interval graphs.
An {\it interval representation} of a graph is an assignment of intervals in the real line to the vertices of the graph so that two vertices are adjacent if and only if their intervals intersect.
A graph is an {\it interval graph} if it has an interval representation.
In~\cite{cheng09}, Cheng developed an enumerative algorithm to determine the distinguishing number of an interval graph, and in this paper we adapt this algorithm to prove the following theorem.
\begin{theorem}\label{thm:equality}
	If $G$ is an interval graph, then $D_{\ell}(G) = D(G)$.
\end{theorem}
Specifically, we will prove that if $G$ is an interval graph and $L=\{L(v)\}_{v\in V(G)}$ is a list assignment in which every list has size $k$, then there are at least as many distinguishing $L$-colorings of $G$ as there are distinguishing $k$-colorings.
Thus $D_{\ell}(G)\le D(G)$.

In Section~\ref{sec:preliminaries}, we establish terminology and notation, and outline Cheng's algorithm for determining the distinguishing number of an interval graph.
This includes a discussion of PQ-trees, a data structure developed by Booth and Lueker~\cite{booth76} that is used in a classification algorithm for interval graphs~\cite{booth79}.
In Section~\ref{sec:proof}, we prove Theorem~\ref{thm:equality}.

\section{Preliminaries}\label{sec:preliminaries}

Let $G$ be a graph.
Two colorings $c$ and $c'$ of $G$ are {\it equivalent} if there is an isomorphism $\varphi$ of $G$ such that $c(v) = c'(\varphi(v))$ for all $v\in V(G)$.
Let $D(G;k)$ denote the number of equivalence classes of distinguishing $k$-colorings of $G$.
Note that $D(G)$ is the minimum positive integer $k$ for which $D(G;k)>0$.
Similarly, given a list assignment $L=\{L(v)\}_{v\in V(G)}$, let $D(G;L)$ denote the number of equivalence classes of distinguishing $L$-colorings of $G$.

Cheng's algorithm computes $D(G;k)$ when $G$ is an interval graph by recursively computing $D(G;k)$.
This algorithm uses PQ-trees, which were defined by Booth and Lueker in~\cite{booth76}. Throughout this paper we will be working with a graph $G$ and a corresponding PQ-tree.
For clarity we will refer to the vertices of $G$ as vertices and the vertices of the PQ-tree as {\it nodes}. 
Formally, a {\it PQ-tree} is a rooted ordered tree whose non-leaf nodes are classified as P-nodes or Q-nodes.
It is {\it proper} if every P-node has at least two children and every Q-node has at least three children. 
A {\it transformation} of a PQ-tree is a permutation of the order of the children of a P-node or the reversal of the order of the children of a Q-node.
Two PQ-trees that differ only by a sequence of transformations are {\it equivalent}.
The {\it frontier} of $T$ is the ordering of its leaves read from left to right, and the {\it frontier} of a node $x$ in $T$ is the frontier of the subtree rooted at $x$.
We let $T_x$ denote the subtree of $T$ rooted at $x$ that includes $x$ and all its descendants.

To build a PQ-tree from an interval graph $G$, we exploit the following characterization of interval graphs due to Fulkerson and Gross~\cite{FG}.
An ordering of the maximal cliques of a graph $G$ satisfies the {\it consecutiveness property} if for every vertex $v$, the maximal cliques that contain $v$ appear consecutively in the ordering.
\begin{theorem}\label{thm:IntervalGrChar}(Fulkerson-Gross~\cite{FG})
A graph $G$ is an interval graph if and only if there is an ordering of the maximal cliques of $G$ that satisfies the consecutiveness property.
\end{theorem}

Given a graph $G$, a {\it PQ-tree for $G$} is a PQ-tree $T$ such that a) the leaves of $T$ correspond to the maximal cliques of $G$ and b) the set of the frontiers of all PQ-trees that are equivalent to $T$ is the set of all orderings of the maximal cliques of $G$ that satisfy the consecutiveness property.
It follows from Theorem~\ref{thm:IntervalGrChar} that if $G$ has a $PQ$-tree, then $G$ is an interval graph.
Booth and Lueker proved the converse holds as well.
\begin{theorem}\label{thm:BLPQ-tree}(Booth-Lueker~\cite{booth76})
If $G$ is an interval graph, then a proper PQ-tree for $G$ can be constructed and it is unique up to a sequence of transformations.
\end{theorem}

The proper PQ-trees of a graph $G$ form an equivalence class under transformations, and we use $T(G)$ to denote a representative of this equivalence class.
It is straightforward to check that if $G$ and $G'$ are isomorphic, then $T(G)$ and $T(G')$ are also isomorphic (see~\cite{cheng09} for details).
However, there are nonisomorphic graphs $G$ and $G'$ such that $T(G)$ and $T(G')$ are isomorphic.
To ensure that PQ-trees uniquely determine a corresponding interval graph, Lueker and Booth~\cite{booth79} later introduced a scheme for labeling the nodes of PQ-trees.
For each vertex $v\in V(G)$, the {\it characteristic node} of $v$, denoted $char(v)$, is the deepest node in $T(G)$ whose frontier contains all maximal cliques of $G$ that contain $v$. 
For each node $x$ in $T(G)$, we let $char^{-1}(x)$ denote the preimage of $x$ under the function $char$, that is, the set of vertices in $G$ whose characteristic node is $x$.
Given $x\in T(G)$, label the children of $x$ as $y_1,\ldots,y_h$ reading from left to right.
Lueker and Booth proved that for each vertex $v\in char^{-1}(x)$, there is a set of consecutive indices $span(v)=[i,j]=\{i,i+1,\ldots,j\}$ such that the union of the frontiers of $y_i,\ldots,y_j$ is exactly the set of maximal cliques in $G$ that contain $v$.
Furthermore, if $x$ is a P-node, then every vertex in $char^{-1}(x)$ has span $[1,h]$.
Each node $x$ of $T(G)$ is then labeled with $|char^{-1}(x)|$ if $x$ is a P-node and the multiset of the spans of the vertices in $char^{-1}(x)$ if $x$ is a Q-node.

Colbourn and Booth proved that, given the additional information of characteristic nodes and spans, the labeled PQ-tree of an interval graph $G$ determines the automorphisms of $G$.
\begin{theorem}(Colbourn-Booth~\cite{colbourn81})
Let $G$ be an interval graph and let $T(G)$ be its PQ-tree.
Every atuomorphism of $T(G)$ induces a distinct automorphism on $G$.
Conversely, every automorphism of $G$ is completely determined by an automorphism of $T(G)$ together with a permutation of the vertices with the same characteristic nodes and spans.
\end{theorem}

To determine the distinguishing number of interval graphs, Cheng introduced terminology to describe how vertices with the same characteristic node are related.
Let $G$ be an interval graph.
Let $x$ be a Q-node in $T(G)$ and assume that $x$ has $h$ children.
Two vertices $a$ and $a'$ in $char^{-1}(x)$ are {\it clones} if $span(a) = span(a')$.
They are {\it twins} if $span(a) = [i_a, j_a]$ and $span(a') = [h+1-j_a, h+1-i_a]$. 
We say that $a$ is the {\it older twin} if $i_a \geq h+1-j_a$. 
A set $A$ is a {\it representative set} for $char^{-1}(x)$ if it is a smallest subset of $char^{-1}(x)$ that contains a clone of every $a \in char^{-1}(x)$.
A subset $A'$ of $A$ is a {\it subrepresentative set} for $A$ if it is a smallest subset of $A$ that contains a clone or older twin of every vertex in $char^{-1}(x)$.

Given a node $x\in T(G)$, let $G_x$ denote the subgraph of $G$ that is induced by the vertices whose characteristic nodes lie in the subtree of $T(G)$ rooted at $x$.
Cheng's algorithm for computing $D(G;k)$ for an interval graph $G$ consists of first constructing the PQ-tree $T(G)$.
The algorithm then completes a postorder traversal of $T(G)$, computing the number of distinguishing $k$-colorings of $G_x$ at each node using the following two theorems. 
The first applies when the root of $T(G)$ is a P-node.
For our purposes, we treat the leaves of $T(G)$ as P-nodes with no children.

\begin{theorem}\label{lemma:pnode}(Cheng~\cite{cheng09}) 
Let $G$ be an interval graph, let $T(G)$ be its PQ-tree, and let $r$ be the root of $T(G)$.
Suppose that $r$ is a P-node, there are $s$ isomorphism classes in $\mathcal G=\{G_x\st x \text{ is a child of }r\}$, and the $i$th isomorphism class contains $m_i$ copies of $G_{x_i}$.
Let $n_{r}=|char^{-1}(r)|$.
The number of distinguishing $k$-colorings of $G$ is given by 
\[ D(G;k) = \binom{k}{n_{r}} \prod_{i=1}^s \binom{D(G_{x_i}; k)}{m_i}. \] 
\end{theorem}

Now assume that the root $r$ of $T(G)$ is a Q-node, and let $x_1,x_2,\ldots,x_h$ be the children of $r$ ordered from left to right.
There are two types of automorphism of $G$: those that map $G_{x_i}$ to itself for all $i\in[h]$, and those that map $G_{x_i}$ to $G_{x_{h+1-i}}$ for all $i\in[h]$.
If there are automorphism of the second type, then we say that $r$ is {\it reversible}; otherwise $r$ is {\it not reversible}.

\begin{theorem}\label{lemma:qnode}(Cheng~\cite{cheng09})
Let $G$ be an interval graph, let $T(G)$ be its PQ-tree, and let $r$ be the root of $T(G)$.
Suppose that $r$ is a Q-node and let $x_1, x_2,\ldots,x_{h}$ be the children of $r$ ordered from left to right. 
Let $A$ be a representative set of $char^{-1}(r)$ and let $A'$ be a subrepresentative set of $A$.
For each $a\in A$, let $m_a$ denote the number of clones of $a$.
If $r$ is not reversible, then the number of distinguishing colorings of $G$ is given by
\[D(G;k) = \prod_{a \in A} \binom{k}{m_a} \prod_{i=1}^h D(G_{x_i};k).\]
If $r$ is reversible, then the number of distinguishing colorings of $G$ is given by
\[D(G;k) = \frac{1}{2} \left(\prod_{a \in A} \binom{k}{m_a} \prod_{i=1}^{h} D(G_{x_i};k) - \prod_{a \in A'} \binom{k}{m_a} \prod_{i=1}^{\CL{h/2}} D(G_{x_i};k)\right).\]
\end{theorem}

\section{Proof of Theorem~\ref{thm:equality}}\label{sec:proof}

Given an interval graph $G$ and a list assignment $L = \{L(v)\}_{v \in V(G)}$, let $D(G;L)$ be the the number of equivalence classes of distinguishing $L$-colorings of $G$.
When $G'$ is an induced subgraph of $G$, we let $D(G';L)$ denote the number of equivalence classes of distinguishing $L$-colorings of $G'$ where the lists at the vertices of $G'$ come from the restriction of the list assignment $L$ to $V(G')$.
We follow Cheng's proofs from~\cite{cheng09} of Theorems~\ref{lemma:pnode} and~\ref{lemma:qnode}  to prove that $D(G; L) \geq D(G;k)$. 
We first prove that Theorem~\ref{thm:equality} holds for complete graphs, which we will use several times in the general proof.

\begin{lemma}\label{lemma:complete}
Let $n\ge 1$.
If $L$ is a list assignment on $K_n$ with $|L(v)| = k$ for all $v \in V(K_n)$, then $D(K_n;L)\ge D(K_n;k)=\binom{k}{n}$.
\end{lemma}

\begin{proof} 
Let $V(K_n)=\{v_1, v_2,\ldots,v_n\}$.
In a distinguishing coloring of $K_n$, all vertices have distinct colors.
Hence $D(K_n;k)=\binom kn$.
Now we count the equivalence classes of distinguishing $L$-colorings.
If we color the vertices in order according to their index, then there are at least $k+1-i$ colors available when we color $v_i$.
There are at most $n!$ colorings that are equivalent to any given coloring, so 
\[ D(K_n;L) \geq \frac{k(k-1)\cdots(k-n + 1)}{n!} = \binom{k}{n} = D(K_n;k).\qedhere\]
\end{proof}

We are now ready to prove Theorem~\ref{thm:equality} in general.

\begin{proof}[Proof of Theorem~\ref{thm:equality}]
Let $G$ be an interval graph and let $L$ be a list assignment on $G$ with $|L(v)| = k$ for all $v \in V(G)$.
Let $T=T(G)$ be the PQ-tree of $G$ and let $r$ be the root of $T$.
We proceed by induction on the number of nodes in $T$ to prove that $D(G;L)\ge D(G;k)$.

If $T$ has only one node, then $G$ has only one maximal clique and therefore $G$ is a complete graph. 
By Lemma~\ref{lemma:complete}, it follows that $D(G;L) \geq D(G;k)$.

Now assume $T$ has at least two nodes.
Let $h$ denote the number of children of $r$; observe that $h\ge 1$ because $T$ is proper by definition.
We consider two cases depending on whether $r$ is a P-node or a Q-node.

{\bf Case 1.} Suppose that $r$ is a P-node. 
Partition the children of $r$ into classes so that two nodes $x$ and $y$ are in the same class if and only if $G_x$ is isomorphic to $G_y$. 
Let $C_1,\ldots,C_s$ denote the equivalence classes of the children of $r$ and label the nodes in class $C_i$ as $x_{i,1}, x_{i,2},\ldots,x_{i,m_i}$ where $m_i = |C_i|$.
Each equivalence class of distinguishing $L$-colorings of $G$ consists of 
\begin{enumerate}
\item an equivalence class of distinguishing $L$-colorings of the subgraph of $G$ induced by $char^{-1}(r)$, and
\item an equivalence class of distinguishing $L$-colorings of $\bigcup_{j=1}^{m_i}G_{x_{i,j}}$ for each $i\in [s]$.
\end{enumerate}

Let $|char^{-1}(r)|=n_r$.
Since $G[char^{-1}(r)]=K_{n_{r}}$, by Lemma~\ref{lemma:complete} there are at least $\binom{k}{n_{r}}$ equivalence classes of distinguishing $L$-colorings of the subgraph of $G$ induced by $char^{-1}(r)$.

Since $r$ is a $P$-node, any nontrivial permutation of children of $r$ that lie in the same equivalence class yields a nontrivial automorphism of $G$.
Since the subtrees $T_{x_{i,j}}$ are isomorphic for all $i\in [s]$ and $j\in[m_i]$, the subgraphs $G_{x_{i,j}}$ for $j\in [m_i]$ must be assigned colorings from $m_j$ distinct equivalence classes.
By induction, $D(G_{x_{i,j}}; L) \geq D(G_{x_{i,j}}; k)$ for all $j \in [m_i]$. 
Set $d_i = D(G_{x_{i,1}}; k)$, and note that $d_i = D(G_{x_{i,j}}; k)$ for all $j\in [m_i]$.
Choosing the colorings in order according to the index $j$, there are at least $d_i - j + 1$ available distinguishing colorings for $G_{x_{i,j}}$.
Thus there are at least $d_i(d_i-1)\cdots(d_i-m_i + 1)$ ways to assign $m_i$ pairwise inequivalent distinguishing colorings to the graphs $G_{x_i,1},\ldots,G_{x_i,m_i}$.
Each choice of the $m_i$ colorings is counted at most $m_i!$ times.
Letting $x_i=x_{i,1}$, it follows that there are at least
\[\frac{d_i(d_i-1)\cdots(d_i-m_i + 1)}{m_i!} = \binom{d_i}{m_i} = \binom{D(G_{x_{i}}; k)}{m_i} \]
equivalence classes of distinguishing $L$-colorings of $\bigcup_{j=1}^{m_i}G_{v_{i,j}}$.
Therefore
\[ D(G; L) \geq \binom{k}{n_{r}} \prod_{i=1}^s \binom{D(G_{x_{i}}; k)}{m_i}\]
and by Theorem~\ref{lemma:pnode},  we conclude that $D(G;L)\ge D(G;k)$.

{\bf Case 2.} Suppose that $r$ is a Q-node.
Let $x_1,\ldots,x_h$ be the children of $r$ ordered from left to right.
Let $A$ be a representative set of $char^{-1}(r)$ and let $A'$ be a subrepresentative set of $A$.
For $a\in A$, let $m_a$ denote the number of clones of $a$ and let $Q_a$ denote the complete subgraph of $G$ induced by the clones of $a$.
There are two cases to consider: when $r$ is not reversible and when $r$ is reversible.

{\bf Case 2.1.} Suppose that $r$ is not reversible.
In this case, each equivalence class of distinguishing $L$-colorings of $G$ consists of 
\begin{enumerate}
\item an equivalence class of distinguishing $L$-colorings of $Q_a$ for each $a\in A$, and
\item an equivalence class of distinguishing $L$-colorings of $G_{x_i}$ for each $i\in [h]$.
\end{enumerate}

By Lemma~\ref{lemma:complete} there are at least $\binom{k}{m_a}$ equivalence classes of distinguishing $L$-colorings of $Q_a$ for each $a\in A$. 
By induction, $D(G_{x_i}; L) \geq D(G_{x_i}; k)$. 
Therefore
\begin{align*} 
D(G;L) &\geq \prod_{a \in A} \binom{k}{m_a} \prod_{i=1}^h D(G_{x_i};L) \\
&\geq \prod_{a \in A} \binom{k}{m_a} \prod_{i=1}^h D(G_{x_i};k)
\end{align*}
By Theorem~\ref{lemma:qnode}, we conclude that $D(G;L)\ge D(G;k)$.

{\bf Case 2.2.} Suppose that $r$ is reversible.
We consider the graph in two parts.
Let the {\it left side} of $G$ contain the clones of $a$ for all $a \in A'$ and the subgraphs $G_{x_i}$, for $i\in[\CL{h/2}]$. 
Let the {\it right side} contain the clones of $a$ for all $a \in A - A'$ and the subgraphs $G_{x_i}$, for $i \in\{\CL{h/2}+1,\ldots,h\}$.  
Since $r$ is reversible, we color the left side first and then color the right side so that any automorphism that reverses the children of $r$ is not color-preserving.

First we color the left side of $G$. 
For $a\in A'$, there are $D(Q_a;L)$ equivalence classes of distinguishing $L$-colorings of $Q_a$.
For $i\in[\CL{h/2}]$, there are $D(G_{x_i};L)$ equivalence classes of distinguishing $L$-colorings of $G_{x_i}$. 
Thus there are $\prod_{a \in A'} D(Q_a;L) \prod_{i=1}^{\CL{h/2}} D(G_{x_i};L)$ equivalence classes of distinguishing $L$-colorings of the left side of $G$.

Next we color the right side of $G$.
For $a\in A-A'$, there are $D(Q_a; L)$ equivalence classes of distinguishing $L$-colorings of $Q_a$.
For $i\in\{\CL{h/2}+1,\ldots,h\}$, there are $D(G_{x_i}; L)$ equivalence classes of distinguishing $L$-colorings of $G_{x_i}$.
Observe that there is at most one equivalence classes of colorings of the right side of $G$ for which a) the clones of $a\in A-A'$ receive the same colors as the clones of the older twin of $a$ and b) the coloring of $G_{x_i}$ and $G_{x_{h+1-i}}$ are equivalent for each $i\in [\FL{h/2}]$ (this is essentially a coloring that mirrors the coloring of the left side).
Therefore there are at least $\left(\prod_{a \in A-A'} D(Q_a; L) \prod_{i=\CL{h/2}+1}^h D(G_{x_i};L)\right) - 1$ equivalence classes of colorings of the right side of $G$ that complete an equivalence class of distinguishing $L$-colorings of $G$.

By Lemma~\ref{lemma:complete}, $D(Q_a; L) \geq \binom{k}{m_a}$ for all $a\in A$, and by induction $D(G_{x_i};L)\ge D(G_{x_i};k)$ for all $i\in [h]$.
Finally, note that each equivalence class of distinguishing $L$-colorings may have been counted twice, once read left to right and once read right to left.
Therefore 
\begin{align*}
D(G;L) &\geq \frac{1}{2} \prod_{a \in A'} D(Q_a; L) \prod_{i=1}^{\CL{h/2}} D(G_{x_i};L)\left( \prod_{a \in A-A'} D(Q_a; L) \prod_{i=\CL{h/2}+1}^h D(G_{x_i};L) - 1\right)\\
& \geq \frac{1}{2} \prod_{a \in A'} \binom{k}{m_a} \prod_{i=1}^{\CL{h/2}} D(G_{x_i};k)\left( \prod_{a \in A-A'} \binom{k}{m_a} \prod_{i=\CL{h/2}+1}^h D(G_{x_i};k) - 1\right)\\
&=  \frac{1}{2} \left(\prod_{a \in A} \binom{k}{m_a} \prod_{i=1}^{h} D(G_{x_i};k) - \prod_{a \in A'} \binom{k}{m_a} \prod_{i=1}^{\CL{h/2}} D(G_{x_i};k)\right).
\end{align*}
Applying Theorem~\ref{lemma:qnode}, we conclude that $D(G;L)\ge D(G;k)$.
\end{proof}

\end{document}